\documentclass[reqno]{amsart}
\usepackage{mathptmx}
\usepackage{pxfonts}
\usepackage{amssymb}
\usepackage{amsfonts}
\usepackage{amsmath}
\usepackage{graphicx}
\usepackage{shadow}
\usepackage{color,enumerate}

\usepackage[all]{xy}
\usepackage[pagebackref]{hyperref}
\newtheorem{thm}{Theorem}[section]
\newtheorem{corollary}[thm]{Corollary}
\newtheorem{lemma}[thm]{Lemma}
\newtheorem{proposition}[thm]{Proposition}

\theoremstyle{definition}
\newtheorem{definition}[thm]{Definition}
\newtheorem{example}[thm]{Example}
\numberwithin{equation}{thm}

\theoremstyle{remark}         
\newtheorem{remark}[thm]{Remark}

\DeclareMathOperator{\Ker}{Ker}
\DeclareMathOperator{\Ima}{Im}

\DeclareMathOperator{\rank}{rank}

\DeclareMathOperator{\edim}{edim}
\DeclareMathOperator{\cdim}{codim}

\DeclareMathOperator{\rd}{rd}

\DeclareMathOperator{\lr}{\longrightarrow}

\begin{document}

\bibliographystyle{amsplain}

\title[Basically regular local homomorphisms]{Basically regular local homomorphisms}

\author[Samir Bouchiba]{Samir Bouchiba}
\address{Samir Bouchiba, Department of Mathematics, University of Meknes, Meknes 50000, Morocco}
\email{s.bouchiba@fs-umi.ac.ma}

\author[Salah Kabbaj]{Salah Kabbaj}
\address{Salah Kabbaj, Department of Mathematics and Statistics, KFUPM, Dhahran 31261, KSA}
\email{kabbaj@kfupm.edu.sa}

\author{Keri Sather-Wagstaff}
\address{Keri Sather-Wagstaff, School of Mathematical and Statistical Sciences,
Clemson University,
O-110 Martin Hall, Box 340975, Clemson, S.C. 29634
USA}
\email{ssather@clemson.edu}
\urladdr{https://ssather.people.clemson.edu/}

\date{\today}



\dedicatory{Dedicated to the memory of Nick Baeth}

\begin{abstract}
We investigate the transfer of regularity between commutative, noetherian, local rings through a class of local homomorphisms which we call \emph{basically regular}. We give numerical characterizations of these maps, investigate their behavior under composition and decomposition, and compare them with Avramov-Foxby-Herzog's weakly regular local homomorphisms. 
\end{abstract}

\maketitle


\section{Introduction}\label{i}

 Throughout, all rings are commutative and noetherian with identity elements, and ring homomorphisms are unital. Let $(A,m)$ be a  local ring and $I$ a proper ideal of $A$. A  minimal generating set for $I$ is called a \emph{minimal basis}, which corresponds to a vector space basis of $I/m I$ over the field $A/m$. In the important case $I={m}$, the number of elements of a minimal basis is called the \emph{embedding dimension} of $A$ written  $\edim(A)$. The ring $A$ is \emph{regular} if its Krull and embedding dimensions coincide or, equivalently, if it has finite global dimension. The \emph{(embedding) codimension} of $A$ measures its defect of regularity via the formula $\cdim(A):=\edim(A)-\dim(A)$. 

Let $\varphi \colon (A, m,K)\rightarrow (B, n,L)$ be a local homomorphism; that is, a homomorphism of  local rings, which maps $m$ into $n$. Grothendieck  calls $\varphi$ \emph{regular} if it is flat and the \emph{closed fiber ring} $B\otimes_{A} K\cong  B/{mB}$ is geometrically regular  (i.e., $B/{mB}\otimes_{K}F$ is regular for every finite extension $F$ of $ K$ \cite{grothendieck:ega4-4}). Avramov, Foxby and Herzog call $\varphi$ \emph{weakly regular} if it is flat and $B/mB$ is regular \cite{avramov:solh}. Of course, if $\varphi$ is regular, then it is weakly regular.
It is well-known that if $\varphi$ is weakly regular and $A$ is regular then so is $B$; and when $\varphi$ is flat, if $B$ is regular then so is $A$ but $B/{mB}$ need not be~\cite{matsumura:crt}. These descent and ascent results are essentially due  to the  fact that flat local homomorphisms are faithfully flat~\cite{MR979760,MR2284892,MR318138,matsumura:crt}.

Nasseh and Sather-Wagstaff \cite{nasseh:cfwfa} established the following formula comparing the embedding dimensions of rings in a weakly regular local homomorphism: $\edim(B)=\edim(A)+\edim (B/{mB} )$. In view of the classical formula $\dim(B)=\dim(A)+\dim (B/{mB} )$  in this setting~\cite[Theorem 15.1]{matsumura:crt}, one deduces $\cdim(A)=\cdim(B)$, recovering the above fact that $A$ is regular if and only if $B$ is so.

 In this paper, we further investigate the transfer or defect of regularity through local homomorphisms. Most notably, we drop the flatness assumption in the above results and emphasize the behavior of minimal bases which are inherent in the definition of regularity. We call a local homomorphism
 $\varphi \colon  (A, m)\longrightarrow (B, n)$ \emph{basically regular} if any minimal basis of $m$ extends to a subset of a minimal basis of $n$. Comparing to previous notions, we obtain in Corollary~\ref{cor211225a} the second implication:
 \begin{equation}\label{eq211220a}\tag{$\dagger$}
\text{$\varphi$ regular
 $\implies\varphi$ weakly regular
 $\implies\varphi$ basically regular.}
 \end{equation}

In Section~\ref{br}, to facilitate our study of basic regularity, we introduce and investigate a new invariant which measures the defect of regularity of $\varphi$. To this end, we consider the natural linear map of $L$-vector spaces $\overline{\varphi}\colon  m/{m^{2}}\otimes_{K} L\longrightarrow  n/{n^{2}}$ and define the \emph{regular defect} of $\varphi$, denoted $\rd(\varphi)$, to be the nullity of $\overline{\varphi}$. The main result of this section (Theorem~\ref{br:th1}) states that $\varphi$ is basically regular if and only if $\rd(\varphi)=0$ if and only if $\edim(B)=\edim(A)+\edim (B/{mB} )$;  if $\varphi$ is flat, then these conditions are equivalent to $\varepsilon_2(B)=\varepsilon_2(A)+\varepsilon_2 (B/{mB} )$, where $\varepsilon_2$ is the second~deviation. 

Section~\ref{rs} investigates the behavior of the regularity defect under compositions and decompositions. The main result is Theorem~\ref{rs:thm1}, which we apply in~\cite{BKS2}, e.g., to the regular and Cohen factorizations of Avramov, Foxby, and Herzog \cite{avramov:solh}. 

In Section~\ref{bwr}, we further investigate the second implication in~\eqref{eq211220a}.
In particular, we prove in Theorem \ref{bwr:thm1} that under certain circumstances one can reverse said implication:
If $\varphi\colon A\lr B$ is a local homomorphism, then the following conditions are equivalent:
\begin{enumerate}[\rm(1)]
		\item $\varphi$ is basically regular and $B$ is regular;
		\item $A$ and $B/{mB}$ are regular rings and $\dim(B)=\dim(A)+\dim (B/{mB} )$; and
		\item $\varphi$ is weakly regular and $A$ is regular.
	\end{enumerate}

Suitable background on regular rings can be found in \cite{bruns:cmr,eisenbud:ca,grothendieck:ega4-4,MR0345945,matsumura:crt}. Any unreferenced material is standard, as in \cite{MR0345945,matsumura:crt}.

\subsection*{Assumptions and Notation}

Throughout this paper, $\varphi \colon  (A, m, K)\rightarrow (B, n, L)$ denotes a local homomorphism; that is, a homomorphism of  local rings with
$\varphi(m)\subseteq n$ and $K:= A/m$ and $L:= B/n$.
Fix  a proper ideal $I\subset A$, let $\pi_I\colon A\to A/I$ denote the canonical surjection, and let $\varphi_I\colon A/I\to B/IB$ denote the induced map.

Let $\mu_A(I)$ denote the minimal number of generators of $I$. Also, let $\delta_{A}(I)$ denote the  maximal length of a sequence of elements of $I$ which is part of a minimal basis of $m$, and let $\delta^{\varphi}_{B}(I)$ denote the  maximal length of a sequence of elements of $IB$ which is part of a minimal basis of $n$:
\begin{align*}
\mu_A(I)                &=  \dim_K\left({I}/{{m}I}\right)\\
\delta_{A}(I)           &=  \dim_{K}\left((I+m^{2})/{m^{2}}\right)\\
\delta^{\varphi}_{B}(I) &=  \dim_{L}\left(({IB+n^{2}})/{n^{2}}\right)=\delta_B(IB).
\end{align*}
Notice that $\delta_{A}(m)=\mu_{A}(m)=\edim(A)$. Also, note that in \cite{MR3751691}, the quantities $\delta_{A}(I)$ and $\delta^{\varphi}_{B}(I)$ are denoted by $\mu_{A}(I)$ and $\mu^{\varphi}_{B}(I)$, respectively.

The \emph{second deviation} of $A$ is the vector space dimension of the second Andr\'e-Quillen homology group $H_2(A,K,K)$, i.e., $\varepsilon_2(A):= \dim_K(H_2(A,K,K))$. It is known from~\cite{andre:hac} that $A$ is regular (i.e., $\cdim(A)=0$) if and only if $\varepsilon_2(A)=0$,  and from~\cite{MR3771857} that $\varepsilon_2(A)-\cdim(A)$ measures the complete intersection defect of $A$.

\section{Basically Regular Homomorphisms}\label{br}

This section introduces basic regularity and investigates its fundamental properties. The main result of the section is Theorem~\ref{br:th1} which shows that this notion provides a broad framework within which the embedding dimension and the deviation $\varepsilon_2$ (under flatness) exhibit the best possible behavior.

\begin{definition}
\begin{enumerate}[(1)]
\item The map $\varphi$ is said to be \emph{basically regular} if  every minimal basis of $m$ extends to a subset of a minimal basis of $n$. In particular, if $A$ is a field, then $\varphi$ is vacuously basically regular.
\item We define the \emph{regularity defect} of $\varphi$, denoted $\rd(\varphi)$, to be the nullity of $\overline{\varphi}$
where $\overline\varphi$ is the natural linear map of $L$-vector spaces
$\overline{\varphi}\colon   m/{m^{2}}\otimes_KL                  \rightarrow     n/{n^{2}}$
given by $\overline{x}\otimes_K \overline{b} \mapsto        \overline{\varphi(x)b}$,
i.e., the nullity of the natural $K$-linear map $m/m^2\to n/n^2$.
\end{enumerate}
\end{definition}

\begin{example}
The natural flat local map $\varphi\colon  A:=K[\![X^2]\!]\to K[\![X]\!]=:B$ of regular rings is not basically regular since the minimal basis $X^2\in X^2A$ does not extend to 
a subset of a minimal basis of $XB$.
\end{example}

The proof of  Theorem~\ref{br:th1} requires several preparatory results.

\begin{lemma}\label{br:lem0}
Let  $J$ be a proper ideal of $B$ containing $IB$.
Let $x_1, \dots,x_r$ be elements of $I$ such that $\varphi(x_1),\dots,\varphi(x_r)$ is part of a minimal basis of $J$. Then,  $x_1,\dots,x_r$ is part of a minimal basis of $I$.
\end{lemma}

\begin{proof}
Assume that $\sum_{1\leq i\leq r}a_{i}\overline{x_{i}}=0$ in the $K$-vector space $(I+m^2)/{m^{2}}$. In the $L$-vector space $(J+n^2)/{n^{2}}$, this yields
$0=\sum_{1\leq i\leq r}a_{i}\overline{\varphi}(\overline{x_{i}}\otimes 1)=
\sum_{1\leq i\leq r}a_{i}\overline{\varphi(x_{i})}$. It follows that $a_{i}=0$, for each $i=1,\dots,r$ as desired.
\end{proof}

\begin{lemma}\label{br:lem1}
We have
$0\leq \delta_{A}(I)-\delta_B^\varphi(I)\leq\rd(\varphi).$
\end{lemma}

\begin{proof}
The inequality $\delta^{\varphi}_{B}(I)\leq\delta_{A}(I)$ holds because any minimal basis of $I$ extends to a generating sequence of $IB$. Next,  consider the restriction of the $L$-linear map $\overline{\varphi}$ given by
$\overline{\varphi_{I}}\colon ({I+{m}^2})/{{m}^2}\otimes_KL\longrightarrow {n}/{{n}^2}$. Then
$$\Ker(\overline{\varphi_{I}})=\Ker(\overline{\varphi})\cap\left(({I+{m}^2})/{{m}^2}\otimes_KL\right)$$
and
$$\Ima(\overline{\varphi_{I}})=({IB+{n}^2})/{{n}^2}.$$
It follows that
\begin{align*}
\delta_A(I) &=      \dim_K\left(({I+{m}^2})/{{m}^2}\right)\\
            &=      \dim_L\left(({I+{m}^2})/{{m}^2}\otimes_KL\right)\\
            &=      \dim_L\left(({IB+{n}^2})/{{n}^2}\right) +\dim_L\left(\Ker(\overline{\varphi_{I}})\right)\\
            &\leq   \delta_B^\varphi(I)+\dim_L\left(\Ker(\overline{\varphi})\right)
\end{align*}
and thus $\delta_A(I)-\delta_B^\varphi(I)\leq\rd(\varphi)$, as desired.
\end{proof}

\begin{lemma}\label{br:lem1.1}
We have
\begin{align*}
\delta_A(I)&=\edim(A)-\edim (A/I ) &\text{and}&&\delta_B^\varphi(I)&=\edim(B)-\edim (B/{IB} ).
\end{align*}
\end{lemma}

\begin{proof}
The second equality follows  from the  exact sequence of
$L$-vector spaces
$$0\longrightarrow ({IB+{n}^2})/{{n}^2}\longrightarrow {n}/{{n}^2}\longrightarrow 
{n}/({IB+{n}^2})\cong \dfrac {{n}/IB}{({n}/IB)^2}\longrightarrow 0.$$
The first equality is the special case $A=B$.
\end{proof}

The following corollary exhibits additive behavior of  $\delta$. See Corollary \ref{rs:cor5} for a version of this result for local homomorphisms.

\begin{corollary}\label{br:cor4} Let $I\subseteq J$ be proper ideals of $A$. Then $\delta_A(J)=\delta_A(I)+\delta_{ A/I} (J/I ).$
	
\end{corollary}

\begin{proof} Applying Lemma \ref{br:lem1.1}, we get 
\begin{align*} \delta_{ A/{I}}  (J/{I} )&=\edim  (A/{I} )-\edim \left (\dfrac {A/I}{J/I}\right )
	= \edim  (A/{I} )-\edim  ({A}/{J} )
	=\delta_A(J)-\delta_A(I)
	\end{align*}
as desired.
\end{proof}

\begin{corollary}\label{br:cor5}
	Let $a_1,a_2,\cdots,a_r\in I$ be  part of a minimal basis of $m$. Then $$\delta_A(I)=r+\delta_{ A/{(a_1,a_2,\cdots,a_r)}}  (I/{(a_1,a_2,\cdots,a_r)} ).$$
\end{corollary}

\begin{proof} This follows from Corollary \ref{br:cor4} as $\delta_A((a_1,a_2,\cdots,a_r))=r$.
	\end{proof}

Next, we return to the regularity defect.

\begin{proposition}\label{br:lem2}
One has 
\begin{align*}
\rd(\varphi)    &=  \delta_{A}(m)-\delta^{\varphi}_{B}(m)
                =  \edim(A)+\edim (B/{{m}B} )-\edim(B).
\end{align*} 
Moreover, if $\varphi$ is flat, then $$\rd(\varphi)=\varepsilon_2(A)+\varepsilon_2 (B/{mB} )-\varepsilon_2(B).$$
\end{proposition}

\begin{proof}
First note that $\Ima(\overline{\varphi})=({{m}B+{n}^2})/{{n}^2}$. Therefore, we have
\begin{align*}
\rd(\varphi)    &=  \dim_L(\Ker(\overline{\varphi}))\\
                &=  \dim_L\left({m}/{{m}^2}\otimes_KL\right)-\dim_L(\Ima(\overline{\varphi}))\\
                &=
                \dim_K\left({m}/{{m}^2}\right)-\dim_L\left(({{m}B+{n}^2})/{{n}^2}\right)\\
                &=  \delta_{A}(m)-\delta^{\varphi}_{B}(m).
\end{align*}
Further,  Lemma~\ref{br:lem1.1} implies that $\delta_B^\varphi(m)=\edim(B)-\edim (B/{{m}B} )$.

Now, assume that $\varphi$ is flat. Set $\overline{B}:=B/{mB}$ and consider the ring homomorphisms
$B\rightarrow\overline{B}\rightarrow L$.
We get the following portion of the Jacobi-Zariski exact sequence
$$H_2(B,\overline B,L)\rightarrow H_2(B,L,L)\rightarrow H_2(\overline B, L,L)\rightarrow H_1(B,\overline B,L)\rightarrow H_1(B,L,L)\rightarrow H_1(\overline B,L,L)\rightarrow 0.$$
Observe that $H_2(B,\overline B,L)\cong H_2(A,K,L)$ by flat base change. Also, since $\varphi$ is flat, the $L$-linear map $H_2(A,K,L)\rightarrow H_2(B,L,L)$ is injective \cite[Remarks 1.4]{MR485836}. Hence, the map $H_2(B,\overline B,L)\rightarrow H_2(B,L,L)$ is injective. Therefore, we get the next exact sequence
$$0\!\rightarrow H_2(B,\overline B,L)\rightarrow H_2(B,L,L)\rightarrow H_2(\overline B, L,L)\rightarrow H_1(B,\overline B,L)\rightarrow H_1(B,L,L)\rightarrow H_1(\overline B,L,L)\!\rightarrow \! 0$$
By flat base change, we have
$$\dim_L(H_i(B,\overline B,L))=\dim_K(H_i(A,K,K))$$
for $i=1,2$. 
Also, we have $\dim_L(H_1(B,L,L))=\edim(B)$, and similarly for $\edim(\overline B)$ and $\edim(\overline A)$.
It follows that
$$\varepsilon_2(B)+\edim(A)+\edim(\overline B)=\varepsilon_2(A)+\varepsilon_2(\overline B)+\edim(B).$$
Using the above equality, we obtain
$$\rd(\varphi)=\edim(A)+\edim(\overline B)-\edim(B)=\varepsilon_2(A)+\varepsilon_2(\overline B)-\varepsilon_2(B)$$ completing the proof.
\end{proof}

Next, we quickly record an immediate consequence of Proposition \ref{br:lem2}.

\begin{corollary}\label{br:cor1}
	If $mB=n$, e.g., if $\varphi$ is surjective, then $\rd(\varphi)=\edim(A)-\edim(B).$
\end{corollary}

\begin{lemma}\label{br:lem3}
The quantity $\delta_B^\varphi(I)$ is equal to the maximal length of a sequence of elements of $I$ which extends to a subset of a minimal basis of $n$.
\end{lemma}

\begin{proof}
Let $t$ denote the maximal length of a sequence of elements of $I$ which extends to a subset of a minimal basis of $n$. First, note that $t\leq\delta_B^\varphi(I)$ by definition of the delta invariant. Next, observe that any  element of the $L$-vector space $({IB+{n}^2})/{{n}^2}$ is a finite combination of elements of the form $\overline{b\varphi(x)}=\overline{b}\ \overline{\varphi(x)}$ with $b\in B$ and $x\in I$. This means that $({IB+{n}^2})/{{n}^2}$ is spanned by $\{\overline{\varphi(x)}\mid x\in I\}$ over $L$. It follows that $\delta_B^\varphi(I)\leq t$, proving the lemma.
\end{proof}

\begin{lemma}\label{br:lem4}
If some minimal basis of $m$ extends to a subset of a minimal basis of $n$, then $\varphi$ is basically regular.
\end{lemma}

\begin{proof}
Let $y_1,\dots,y_r$ be a minimal basis of ${m}$ such that $\varphi(y_1),\dots,\varphi(y_r)$ are part of a minimal basis of $n$. Then  $\{\overline{y_1}\otimes_K1,\dots,\overline{y_r}\otimes_K1\}$ is a basis for the $L$-vector space $m/{{m}^2}\otimes_KL$ and $\rank(\overline{\varphi})=r$. Now, let $x_1,\dots,x_r$ be another minimal basis of $m$. Then $\{\overline{x_1}\otimes_K1,\dots,\overline{x_r}\otimes_K1\}$ is a basis for the $L$-vector space $m/{{m}^2}\otimes_KL$ and thus $\dim_{L}\langle\overline{\varphi}(x_i\otimes_K1)\mid{1\leq i\leq r}\rangle=\rank(\overline{\varphi})=r$. It follows that $\{\overline{\varphi(x_1)},\dots,\overline{\varphi(x_r)}\}$ is a linearly independent subset of $n/{n^{2}}$; that is, $\varphi(x_1),\dots,\varphi(x_r)$ are part of a minimal basis of $n$, as desired.
\end{proof}

Here is the main result of this section. It shows that basic regularity provides a broad framework within which the embedding dimension and the deviation $\varepsilon_2$ (under flatness) exhibit the best possible behavior. 

\begin{thm}\label{br:th1}
The following conditions are equivalent:
\begin{enumerate}[\rm(1)]
\item $\varphi$ is basically regular;
\item $\delta_{A}(I)=\delta^{\varphi}_{B}(I)$ for each proper ideal $I$ of $A$;
\item $\rd(\varphi)=0$;
\item $\edim(B)=\edim(A)+\edim (B/{mB} )$.
\end{enumerate}
Moreover, if $\varphi$ is flat, then the above conditions are equivalent to:
\begin{enumerate}[\rm(5)]
\item $\varepsilon_2(B)=\varepsilon_2(A)+\varepsilon_2 (B/{mB} )$.
\end{enumerate}
\end{thm}

\begin{proof}
(1) $\Rightarrow$ (2) Let $I$ be a proper ideal of $A$ and let $x_1,\dots,x_r$ be elements of $I$ which are part of a minimal basis of $m$. By (1), $\{\varphi(x_1),\dots,\varphi(x_r)\}$ is a subset of a minimal basis of $n$. Hence, $\delta_A(I)\leq\delta_B^\varphi(I)$. But, by
Lemma~\ref{br:lem1}, we always have $\delta_B^\varphi(I)\leq\delta_A(I)$, and the
equality follows.

(2) $\Rightarrow$ (3) Assume that (2) holds. Then, in particular, $\delta_A({m})=\delta_B^\varphi({m})$. By Proposition~\ref{br:lem2}, we have $\rd(\varphi)=0$, as desired.

(3) $\Rightarrow$ (1) Assume that $\rd(\varphi)=0$ and let $r:=\edim(A)$. Then, by Proposition~\ref{br:lem2},  $\delta_B^\varphi({m})=r$. Hence, by Lemma~\ref{br:lem3}, there exist $x_1,\dots,x_r\in {m}$ such that $\varphi(x_1),\dots,\varphi(x_r)$ is part of a minimal basis of $n$.  Lemma~\ref{br:lem0} implies that $x_1,\dots,x_r$ is a minimal basis of $m$, so $\varphi$ is basically regular by Lemma~\ref{br:lem4}.

(3) $\Leftrightarrow$ (4) This is handled by Proposition~\ref{br:lem2}.

(3) $\Leftrightarrow$ (5) If $\varphi$ is flat, then 
$\rd(\varphi)=\varepsilon_2(A)+\varepsilon_2 (B/{mB} )-\varepsilon_2(B)$ by Proposition \ref{br:lem2}, completing the proof.
\end{proof}

Next, we list a few consequences of Theorem~\ref{br:th1}. 

\begin{corollary}\label{cor211225a} 
If $\varphi$ is weakly regular,
then it is also basically regular.
\end{corollary}

\begin{proof}
Assume that $\varphi$ is weakly regular.
Then $\edim(B)=\edim(A)+\edim ( B/{mB} )$ by~\cite[Lemma~3.1]{nasseh:cfwfa}. Thus, Theorem \ref{br:th1} implies that  $\varphi$ is basically regular. 
\end{proof}

It is reasonable to ask when the converse of the preceding corollary holds. 
The next result addresses this in the flat case. This is pursued further in Section~\ref{bwr} below.

\begin{corollary}\label{brw:cor3} 
	The following conditions are equivalent:
	\begin{enumerate}[\rm (1)]
		\item $\varphi$ is flat and basically regular with $\varepsilon_2(A)=\varepsilon_2(B)$;
		\item $\varphi$ is flat and basically regular with $\cdim(A)=\cdim(B)$;
		\item $\varphi$ is weakly regular.
	\end{enumerate}
\end{corollary}

\begin{proof}
By definition and Corollary~\ref{cor211225a}, we assume without loss of generality that $\varphi$ is flat and basically regular, therefore $\rd(\varphi)=0$.
The standard dimension equality $\dim(B/mB)=\dim(B)+\dim(A)$ for flat maps,  with Proposition~\ref{br:lem2}, implies that
\begin{align*}
\cdim(B/mB)&=\cdim(B)-\cdim(A)\\
\varepsilon_2(B/mB)&=\varepsilon_2(B)-\varepsilon_2(A)
\end{align*}
Since $B/mB$ is regular if and only if $\cdim(B/mB)=0$ if and only if $\varepsilon_2(B/mB)=0$, the result now follows.
\end{proof}

\begin{corollary}\label{br:cor2} 
	\begin{enumerate}[\rm (i)]
	\item $\rd(\pi_I)=\edim(A)-\edim ( A/I )=\delta_A(I).$
	\item The following conditions are equivalent:
	\begin{enumerate}[\rm(1)]
		\item $\pi_I$ is basically regular;
		\item $\edim(A)=\edim ( A/I )$;
		\item $I\subseteq m^{2}$.
	\end{enumerate}
	\item $\rd(\pi_m)=\edim(A)$, where $\pi_m\colon A\lr K= A/m$ is the canonical surjection.
	\end{enumerate}
\end{corollary}

\begin{proof} (i) This follows from the combination of Corollary \ref{br:cor1} and Lemma \ref{br:lem1.1}.

	(ii) This follows  from (i) and Theorem \ref{br:th1} since $ \rd(\pi_I)=0$ if and only if $I\subseteq m^2$.
	
	(iii) This also follows directly from (i).
\end{proof}

\begin{corollary}\label{br:cor3} 
\begin{enumerate}[\rm (1)]
	\item One has
	$\rd(\pi_I)\geq \rd(\pi_{IB}).$
	\item Then $\varphi$ is basically regular if and only if $\rd(\pi_I)=\rd(\pi_{IB})$ for each proper ideal $I$ of $A$.
	\end{enumerate}
\end{corollary}

\begin{proof} (1) This comes from Lemma \ref{br:lem1} and Corollary \ref{br:cor2} as $\delta_B^{\varphi}(I)=\delta_B(IB)=\rd(\pi_{IB})$.

	(2) This is a direct consequence of Theorem \ref{br:th1} and Corollary~\ref{br:cor2}.
	\end{proof}

\section{Compositions and Decompositions}\label{rs}

This section investigates the behavior of the regularity defect of local ring homomoprhisms within commutative diagrams. The main result here is Theorem~\ref{rs:thm1}, which we apply subsequently~\cite{BKS2}, e.g., to the regular and Cohen factorizations of Avramov, Foxby, and Herzog \cite{avramov:solh}.  

\begin{definition} 
An \emph{oriented commutative diagram of local homomorphisms} is a commutative diagram $\mathcal T$ of local homomorphisms 
equipped with an orientation, either clockwise or anticlockwise. 
In practice, we simply write that the diagram is 
a commutative diagram of local homomorphisms with the clockwise or anticlockwise orientation.
\end{definition} 

\begin{definition} Consider the next commutative triangle of local homomorphisms 
\begin{equation}
\begin{split}
\xymatrix{ A\ar[r]^\varphi\ar[rd]_{\psi\circ\varphi}&B\ar[d]^\psi\\
	&C.}
\end{split}
\label{diag211222b}\tag{$\mathcal{T}$}
\end{equation}
	\begin{enumerate}[1)]
\item If $\mathcal T$ is oriented with the clockwise orientation, the \emph{regularity defect} of $\mathcal T$ is defined to be the integer $$\rd(\mathcal T):=(\rd(\varphi)+\rd(\psi))-\rd(\psi\circ\varphi).$$ If $\mathcal T$ is oriented in the anticlockwise direction, then $$\rd(\mathcal T):=\rd(\psi\circ\varphi)-(\rd(\varphi)+\rd(\psi)).$$ 
\item In either orientation, the triangle $\mathcal T$ is \emph{basically regular} if $\rd(\mathcal T)=0$, that is, if $\rd(\psi\circ \varphi)=\rd(\varphi)+\rd(\psi)$.
Note that this is independent of the orientation.
	\end{enumerate}
Analogous definitions are made for other oriented commutative diagrams.
\end{definition} 

\begin{remark} \label{rmk211223a}
Consider an oriented commutative square of  local  homomorphisms
\begin{equation}
\begin{split}
\xymatrix{
	A\ar[r]^\varphi\ar[d]_{\varphi^{\prime}} &B\ar[d]^{\psi}\\
	C\ar[r]_{\psi^{\prime}}&D.}
\end{split}
\label{diag211222a}\tag{$\mathcal{S}$}
\end{equation}
Then $(\mathcal S)$ consists of the union of the following commutative triangles $(ABD)$ and $(ADC)$ which are oriented in opposite directions \[\xymatrix{
	A\ar[r]^\varphi\ar[rd]^\theta\ar[d]_{\varphi^{\prime}} &B\ar[d]^{\psi}\\
	C\ar[r]_{\psi^{\prime}}&D}
\] where $\theta:=\psi\circ\varphi=\psi^{\prime}\circ\varphi^{\prime}$. It is straightforward to check that if $(\mathcal S)$ has the anticlockwise orientation, then
\begin{align*}\rd(ABD)+\rd(ADC)&=(\rd(\psi)+\rd(\varphi))-(\rd(\psi^{\prime})+\rd(\varphi^{\prime}))=\rd(ABDC)\end{align*}
and similarly if $(\mathcal S)$ has the anticlockwise orientation.
In either case, the regularity defect of a commutative square is the sum of the regularity defects of its two commutative triangles. One may interpret this as measuring how far the regularity defect is preserved along different paths from $A$ to $D$.
\end{remark}

Here is the main theorem of this section. It describes the behavior of the regularity defect of a commutative square under surjective base change.

\begin{thm} \label{rs:thm1} Consider the following commutative square of local ring homomorphisms
	\begin{equation}
	\begin{split}
	\xymatrix{
		A\ar[r]^\varphi\ar[d]_{\varphi^{\prime}} &B\ar[d]^{\psi}\\
		(C,m^{\prime})\ar[r]_{\psi^{\prime}}&D}
	\end{split}\tag{$\mathcal S$}
	\end{equation} with the clockwise orientation, and consider the induced diagram 
		\begin{equation}
	\begin{split}
	\xymatrix{
		A/I\ar[r]^{\varphi_I}\ar[d]_{\varphi^{\prime}_I} &B/IB\ar[d]^{\psi_{IB}}\\
		C/IC\ar[r]_{\psi^{\prime}_{IC}}&D/ID}
		\end{split}\tag{$(A/I)\otimes_A\mathcal S$}
	\end{equation} 
	also with the clockwise orientation.
	Then we have
	$$\rd(\mathcal S)=\rd ((A/I)\otimes_A\mathcal S )=\rd (K\otimes_A\mathcal S )=\rd(\psi_{mB})-\rd(\psi^{\prime}_{mC}).$$
	\end{thm}

\begin{proof} In the following sequence, the first equality is by definition and the second one is from Proposition~\ref{br:lem2} 
\begin{align*}\rd(\mathcal S)&=(\rd(\varphi)+\rd(\psi))-(\rd(\varphi^{\prime})+\rd(\psi^{\prime}))\\
	&=\edim (B/{mB} )+\edim (D/{nD} )-\edim ( C/{mC} )-\edim ( D/{m^{\prime}D} ).\end{align*}
In the next sequence, the first and third steps are by definition
\begin{align*}
\rd (K\otimes_A\mathcal S )
&=(\rd (\varphi_m)+\rd(\psi_{mB}))-(\rd(\varphi^{\prime}_{m})+\rd(\psi^{\prime}_{mC}))\\
&=\rd(\psi_{mB})-\rd(\psi^{\prime}_{mC}) \\
&=\edim ( B/{mB} )+\edim ( D/{nD} )-\edim ( D/{mD} )-\edim ( C/{mC} )\\
&\quad -\edim ( D/{m^{\prime}D} )+\edim ( D/{mD} )\\
&=\edim ( B/{mB} )+\edim ( D/{nD} )-\edim ( C/{mC} )-\edim ( D/{m^{\prime}D} )
\end{align*}
The second step follows from the fact that $A/m$ is a field.
Hence, we have 
$$\rd(\mathcal S)=\rd (K\otimes_A\mathcal S )=\rd(\psi_{mB})-\rd(\psi^{\prime}_{mC}).$$ 
By applying this to the commutative square $(A/I)\otimes_A\mathcal S$, we get 
$$\rd ((A/I)\otimes_A\mathcal S )=\rd\left (\dfrac {A/I}{m/I}\otimes_{A/I}\left (( A/I)\otimes_A\mathcal S\right )\right )=\rd (K\otimes_A\mathcal S ).$$ 
These steps provide the desired equalities.
\end{proof}

\begin{corollary}\label{rs:cor2} With $\mathcal S$  as in Theorem~\ref{rs:thm1}, the following conditions are equivalent. 
	\begin{enumerate}[\rm(1)]
		\item $\mathcal S$ is a basically regular square;
		\item For each proper ideal $I$ of $A$ the induced diagram $\mathcal S_I=(A/I)\otimes_A\mathcal S$ is  basically regular;
		\item The induced diagram $\mathcal S_m=K\otimes_A\mathcal S$ is  basically regular;
		\item $\rd(\psi_{mB})=\rd(\psi^{\prime}_{mC})$.
	\end{enumerate}
\end{corollary}

\begin{corollary}\label{rs:cor1} The commutative diagram 
	\begin{equation}
	\begin{split}
		\xymatrix{
		A\ar[r]^\varphi\ar[d]_{\pi_{I}} &B\ar[d]^{\pi_{IB}}\\
		A/I\ar[r]_{\varphi_I}&B/IB}
		\end{split}\tag{$\mathcal S_I$}
	\end{equation} is basically regular, and thus 
		\begin{equation}
\rd(\varphi_I)=\rd(\varphi)-(\rd(\pi_I)-\rd(\pi_{IB}))\leq \rd(\varphi).\label{eq211223a}
\end{equation}  Consequently, if $\varphi$ is basically regular, then so is $\varphi_I$.
	
\end{corollary}

\begin{proof} Observe that $ K\otimes_A\mathcal S_I$ has the form
\[\xymatrix{
		K\ar[r]^-{K\otimes_A\varphi}\ar[d]_{1_{K}} &B/mB\ar[d]^{1_{B/mB}}\\
		K\ar[r]_-{K\otimes_A\varphi}&B/mB}
	\] which is basically regular by definition. Thus, Corollary \ref{rs:cor2} implies that $\mathcal S_I$ is also basically regular, hence the equality in~\eqref{eq211223a}. The  inequality in~\eqref{eq211223a} follows from Corollary \ref{br:cor3}, which implies that $0\leq \rd(\pi_I)-\rd(\pi_{IB})$.
\end{proof}

\begin{corollary} \label{rs:cor3} Given a commutative triangle of   local ring homomorphisms 
	\begin{equation}
	\begin{split}
\xymatrix{ A\ar[r]^\varphi\ar[rd]_{\theta}&B\ar[d]^\psi\\
		&C}\end{split}\tag{$\mathcal T$}
	\end{equation} with the clockwise orientation, one has $$\rd(\mathcal T)=\rd (( A/I)\otimes_A\mathcal T )=\rd (K\otimes_A\mathcal T )=\rd(\psi_{mB}).$$
\end{corollary}

\begin{proof} Apply Theorem \ref{rs:thm1} with the square
\[\xymatrix{ A\ar[r]^\varphi\ar[rd]_{\theta}\ar[d]_=&B\ar[d]^\psi\\
		A\ar[r]_\theta&C}\]
noting that $\rd(\mathcal S)=\rd(\mathcal T)+\rd(ACA)=\rd(\mathcal T)$; the first equality here is by Remark~\ref{rmk211223a}, and the second one is by definition.
	\end{proof}

\begin{corollary}\label{rs:cor4}
	Let $\psi\colon B\lr C$ be another local ring homomorphism. Then 
\begin{align}\rd(\psi\circ\varphi)&=\rd(\varphi)+\rd(\psi)-\rd(\psi_{mB})=\rd(\varphi)+(\rd(\pi_{mB})-\rd(\pi_{mC})).\label{eq211224a}
\end{align}
Consequently, 	\begin{equation}\rd(\varphi)\leq\rd(\psi\circ\varphi)\leq\rd(\varphi)+\rd(\psi).\label{eq211224b}\end{equation}
\end{corollary}

\begin{proof} Applying Corollary \ref{rs:cor3}, we get $$\rd(\varphi)+\rd(\psi)-\rd(\psi\circ \varphi)=\rd(\mathcal T)=\rd(\psi_{mB})$$ proving the first equality in~\eqref{eq211224a}. For the rest of~\eqref{eq211224a},
	by Corollary \ref{rs:cor1}, we have
	$$\rd(\psi_{mB})=\rd(\psi)+\rd(\pi_{mC})-\rd(\pi_{mB}).$$ 
	Next, the inequalities in~\eqref{eq211224b} follow from~\eqref{eq211224a} since $\rd(\psi)-\rd(\psi_{mB})\geq 0$
	by Corollary \ref{br:cor3}(1), and $\rd(\psi_{mB})\geq 0$ by definition.
	\end{proof}

\begin{corollary}\label{rs:cor5}
	Consider the commutative triangle of   local ring homomorphisms 
	\begin{equation}
	\begin{split}
\xymatrix{ A\ar@{->>}[r]^\varphi\ar[rd]_{\theta}&B\ar[d]^\psi\\
		&C}\end{split}\tag{$\mathcal T$}
	\end{equation}  such that $\varphi$ is surjective. Then $\mathcal T$ is basically regular and thus
$\rd(\psi\circ\varphi)=\rd(\varphi)+\rd(\psi).$
\end{corollary} 

\begin{proof} As $\varphi$ is surjective, the map $\varphi_m$ is an isomorphism. So, the induced triangle $K\otimes_A\mathcal T$ is basically regular, and $\mathcal T$ is basically regular, by Theorem~\ref{rs:thm1}, as in the proof of Corollary~\ref{rs:cor3}. 
\end{proof}

Next, we reprove a previous result.

\begin{proof}[Alternate proof of Corollary \ref{br:cor4}] Consider the commutative triangle
\[\xymatrix{ A\ar[r]^{\pi_I}\ar[rd]_{\pi_{ J/I}}& A/I\ar[d]^{\pi_{J}}\\
		& A/J.}
	\]
Corollary \ref{br:cor2}(i) implies $\delta_A(I)=\rd(\pi_I)$, so Corollary \ref{rs:cor5}
	yields the desired result.
	\end{proof}

The next corollary discusses the basic regularity of compositions. In light of Theorem~\ref{br:th1}, it follows from Corollaries~\ref{rs:cor1}, \ref{rs:cor4}, and  \ref{rs:cor5}.

\begin{corollary}\label{rs:cor6}
	Let $\psi\colon B\to C$ be another local
	homomorphism.
	\begin{enumerate}[\rm(1)]
		\item The following conditions are equivalent:
	\begin{enumerate}[\rm i)]
		\item $\psi\circ\varphi$ is basically regular.
		\item $\varphi$ is basically regular and $\rd(\psi)=\rd(\psi_{mB})$.
		\item $\varphi$ is basically regular and $\rd(\pi_{mB})=\rd(\pi_{mC})$.
	\end{enumerate}
		\item  $\psi$ is basically regular if and only if $\psi_{mB}$ 
		is basically regular and $\rd(\psi\circ\varphi)=\rd(\varphi)$.
		\item \label{item3} $\varphi$ and $\psi$ are basically regular $\Rightarrow$ $\psi\circ\varphi$ is basically regular.
	\item If $\varphi$ is surjective, then the converse to part~\eqref{item3} holds.		
	\end{enumerate}
\end{corollary}

Our next results  build basically regular maps that are not weakly regular from old maps.
The first such result follows readily from Corollaries~\ref{rs:cor1} and~\ref{rs:cor3}.

\begin{proposition}\label{rs:prop1}
	Consider   a proper ideal  $J\subset B$ such that $IB\subseteq J$ and the induced commutative triangle 
\begin{equation}
\begin{split}
\xymatrix{A/I\ar[r]^{\varphi_{I}}\ar[rd]_{\varphi_{I,J}}&B/IB\ar[d]^{\pi_{J/IB}}\\
		&B/J}\end{split}
\tag{$\mathcal{T}$}
\end{equation}
where $\varphi_{I,J}$ is the natural induced local map.
Then $$\rd(\mathcal T)=\rd\left(\pi_{{(J+mB)}/{(mB)}}\right)$$ and thus \begin{align*}\rd(\varphi_{I,J})&=\rd(\varphi_I)+\rd(\pi_{ J/{IB}})-\rd\left(\pi_{{(J+mB)}/{(mB)}}\right)\\
	&=\rd(\varphi)- 
	(\rd(\pi_I)-\rd(\pi_{IB}))+\rd(\pi_{ J/{IB}})-\rd\left(\pi_{{(J+mB)}/{(mB)}}\right).\end{align*}
\end{proposition}

\begin{corollary}\label{rs:cor7}
	Assume that $I\subseteq m^2$, and let $J\subseteq n^2$ be an ideal of $B$ such that $IB\subseteq J$. Consider the induced local homomorphism $\varphi_{I,J}\colon  A/I\lr  B/J$. Then $\rd(\varphi_{I,J})=\rd(\varphi).$
\end{corollary}

\begin{proof} Note that, by Corollary \ref{br:cor2}(ii), one has $\rd(\pi_I)=0=\rd(\pi_{IB})$ as $I\subseteq m^2$ and $IB\subseteq n^2$. Also, $ J/{(IB)}\subseteq  ( {n}/{(IB)} )^2= {(n^2+IB)}/{IB}$ and $ {(J+mB)}/{mB}\subseteq  ( n/{mB} )^2= {(n^2+mB)}/{mB}$. Hence, again by Corollary \ref{br:cor2}(ii), $\rd(\pi_{  J/{IB}})=0=\rd(\pi_{\frac {J+mB}{mB}})$. It follows, by Proposition \ref{rs:prop1}, that $\rd(\varphi_{IJ})=\rd(\varphi)$.	
\end{proof}

The next result follows from Corollary~\ref{rs:cor7} with $J=IB$ and $I=0$, respectively.

\begin{corollary}\label{rs:cor8} 
\begin{enumerate}[\rm (1)]	\item If $I\subseteq m^2$, then $\rd(\varphi_I)=\rd(\varphi)$.
	\item If $J\subseteq n^2$ is an ideal of $B$, then $\rd(\pi_J\circ\varphi)=\rd(\varphi)$.
\end{enumerate}
\end{corollary}

\begin{corollary}\label{rs:cor9}
	Assume that $\varphi$ is basically regular. 
		\begin{enumerate}[\rm (1)]
	\item The induced local homomorphism $\varphi_{I,J}\colon  A/I\lr  B/J$ is basically regular for any ideal $I\subseteq m^2$ of $A$ and any ideal $J\subseteq n^2$ of $B$ such that $IB\subseteq J$.
	\item If $I\subseteq m^2$, then the induced local homomorphism $\varphi_I$ is basically regular.
	\item The composition $\pi_J\circ\varphi\colon A\lr  B/J$ is basically regular for any ideal $J\subseteq n^2$ of $B$.
	\end{enumerate}
\end{corollary}

\begin{proof} It is straightforward by Corollaries \ref{rs:cor7} and  \ref{rs:cor8}.
\end{proof}

\section{Basic Regularity Versus Weak Regularity}\label{bwr}

We begin this section with various ways to construct basically regular homomorphisms which are not weakly regular.

\begin{proposition}\label{rs:cor10} Assume that $\varphi$ is basically regular.
		\begin{enumerate}[\rm (1)]
	\item Assume that $mB\neq n$ and $I\subseteq m^2$. Then the induced map $\varphi_{I,n^2}\colon  A/{I}\lr  B/{n^2}$ is basically regular while the closed fibre ring $\dfrac {B/n^2}{(m/I)B/n^2}$ is not regular.	
	\item Assume that $A$ is not artinian.
	Then $\pi_{n^2}\circ\varphi\colon A\lr  B/{n^2}$ is basically regular and not flat. If, moreover, $mB\neq n$, then the fibre ring $\dfrac {B/n^2}{m(B/n^2)}$ is not regular.
		\end{enumerate}
\end{proposition}

\begin{proof} (1) By Corollary \ref{rs:cor9}(1), the map $\varphi_{I,n^{2}}$ is basically regular. Now, observe that $$\dfrac {B/n^2}{(m/I)B/n^2}\cong  \frac{B}{(mB+n^2)}\cong \dfrac {B/mB}{(n/mB)^2}$$ so $$\edim \left (\dfrac {B/n^2}{(m/I)B/n^2}\right )=\edim  ( B/{mB} )>0=\dim\left (\dfrac {B/n^2}{(m/I)B/n^2}\right )$$ since $ B/{mB}$ is not a field and $\dim ( B/{n^2} )=0$. It follows that  $\dfrac {B/n^2}{(m/I)B/n^2}$ is not  regular.
	 
	(2) By Corollary \ref{rs:cor9}(3), the map $\pi_{n^2}\circ\varphi$ is basically regular. Notice that $$\dim(A)+\dim\left (\dfrac {B/n^2}{m(B/n^2)}\right )>0=\dim ( B/{n^2} )$$ as $A$ is not artinian.  It follows, e.g., by \cite[Theorem 15.1]{matsumura:crt}, that $ B/{n^2}$ is not flat over $A$. If moreover $mB\neq n$, then $\dfrac {B/n^2}{m(B/n^2)}$ is not regular, by (1).
\end{proof}

\begin{corollary}\label{rs:cor11}
	If $A$ is not artinian, then the canonical surjection $\pi_{m^2}\colon A\lr  A/{m^2}$ is basically regular and not flat.
\end{corollary}

For the next result, recall that a \emph{Cohen presentation} of a complete local ring $A$ is a surjection $\pi\colon (Q,r)\twoheadrightarrow  A$ such that $Q$ is regular and complete;
the presentation is \emph{minimal} if $\ker(\pi)\subseteq r^2$. Every complete local ring admits a minimal Cohen presentation by the Cohen structure theorem. 
\begin{corollary}\label{rs:cor11x}
	If $A$ is  complete and not regular, then any minimal Cohen presentation $\pi\colon Q\twoheadrightarrow  A$ is basically regular and not flat.
\end{corollary}

\begin{proof}
Basic regularity of $\pi$ follows, e.g., from Corollary~\ref{br:cor2}(ii).
Since $A$ is not regular, the map $\pi$ is not an isomorphism, so $\ker(\pi)$ contains a $Q$-regular element since $Q$ is an integral domain, implying that $\pi$ is not flat. 
\end{proof}

Recall that a regular local ring $A$ with $\operatorname{char}(K)=p\geq 0$ is  \emph{unramified} if $p=0$ in $A$ or $p\not\in m^2$. The following result records the fact that unramified regular local rings are characterized by the basic regularity of their local structure homomorphism. 

\begin{proposition}\label{rs:prop2}
	Assume that $A$ is regular with  $\operatorname{char}(K)=p\geq 0$, and let $\eta_A\colon \mathbb{Z}_{(p)}\lr A$ be the local structure homomorphism. Then
		\begin{enumerate}[\rm (1)] 
	\item $A$ is unramified if and only if $\eta_A$ is basically regular.
	\item Assume that $A$ is unramified and that $p>0$. Then $\pi_{m^2}\circ\eta_A\colon \mathbb{Z}_{(p)}\lr  A/{m^2}$ is basically regular but is not flat and thus not weakly regular. 
	\end{enumerate}
\end{proposition}

\begin{proof} Part (1) follows from the definitions,
and part	(2) is by Proposition \ref{rs:cor10}.2).
\end{proof}

Next, in the main result of this section, we complement Corollary~\ref{cor211225a}. 
	
	\begin{thm}\label{bwr:thm1}
		The following conditions are equivalent:
		\begin{enumerate}[\rm (1)]
			\item $\varphi$ is basically regular and $B$ is regular;
			\item $A$ and $ B/{mB}$ are regular, and $\dim(B)=\dim(A)+\dim ( B/{mB} )$;
			\item $\varphi$ is weakly regular and $A$ is regular.
		\end{enumerate}
	\end{thm}

	\begin{proof} (3) $\Rightarrow$ (2) This follows from the definition of ``weakly regular'' and  standard  results for flat local maps.
	
		(2) $\Rightarrow$ (1) The first step in the next sequence is by definition.
				\begin{align*}
		0&\leq\rd(\varphi)\\
		&=\edim(A)+\edim (B/{{m}B} )-\edim(B)\\
		&=\dim(A)+\dim (B/{{m}B} )-\edim(B)\\
		&\leq\dim(A)+\dim (B/{{m}B} )-\dim(B)\\
		&=0.
		\end{align*}
		The second step follows from Proposition \ref{br:lem2}.
		The third and fifth steps are by assumption (2), and the fourth step is the standard inequality $\dim(B)\leq\edim(B)$.
		It follows that $\rd(\varphi)=0$ and $\dim(B)=\edim(B)$, as desired.

		(1) $\Rightarrow$ (2) Essentially, run the preceding paragraph in reverse.
		
		(1) $\Rightarrow$ (3) Assuming that (1) holds, we also get that (2) holds by the previous paragraph. In particular, $B$ is Cohen-Macaulay, so we can invoke \cite[Theorem 23.1]{matsumura:crt} to conclude that $\varphi$ is flat. Hence $\varphi$ is weakly regular, by definition, as desired.
		\end{proof}


\begin{thebibliography}{10}

\bibitem{andre:hac}
M.~Andr{\'e}, \emph{Homologie des alg\`ebres commutatives}, Springer-Verlag,
  Berlin, 1974, Die Grundlehren der mathematischen Wissenschaften, Band 206.
  \MR{0352220 (50 \#4707)}

\bibitem{MR485836}
L.~L. Avramov, \emph{Homology of local flat extensions and complete
  intersection defects}, Math. Ann. \textbf{228} (1977), no.~1, 27--37.
  \MR{485836}

\bibitem{avramov:solh}
L.~L. Avramov, H.-B.\ Foxby, and B.\ Herzog, \emph{Structure of local
  homomorphisms}, J. Algebra \textbf{164} (1994), 124--145. \MR{95f:13029}

\bibitem{MR3771857}
S.~Bouchiba, J.~Conde-Lago, and J.~Majadas, \emph{Cohen-{M}acaulay,
  {G}orenstein, complete intersection and regular defect for the tensor product
  of algebras}, J. Pure Appl. Algebra \textbf{222} (2018), no.~8, 2257--2266.
  \MR{3771857}

\bibitem{MR3751691}
S.~Bouchiba and S.~Kabbaj, \emph{Embedding dimension and codimension of tensor
  products of algebras over a field}, Rings, polynomials, and modules,
  Springer, Cham, 2017, pp.~53--77. \MR{3751691}

\bibitem{BKS2} S. Bouchiba, S.~Kabbaj, and K.~Sather-Wagstaff,
Basic regularity of factorizations of local ring homomorphisms,
in preparation.

\bibitem{MR979760}
N.~Bourbaki, \emph{Commutative algebra. {C}hapters 1--7}, Elements of
  Mathematics (Berlin), Springer-Verlag, Berlin, 1989, Translated from the
  French, Reprint of the 1972 edition. \MR{979760}

\bibitem{MR2284892}
\bysame, \emph{\'{E}l\'{e}ments de math\'{e}matique. {A}lg\`ebre commutative.
  {C}hapitres 8 et 9}, Springer, Berlin, 2006, Reprint of the 1983 original.
  \MR{2284892}

\bibitem{bruns:cmr}
W.\ Bruns and J.\ Herzog, \emph{Cohen-{M}acaulay rings}, revised ed., Studies
  in Advanced Mathematics, vol.~39, University Press, Cambridge, 1998.
  \MR{1251956 (95h:13020)}

\bibitem{eisenbud:ca}
D.~Eisenbud, \emph{Commutative algebra}, Graduate Texts in Mathematics, vol.
  150, Springer-Verlag, New York, 1995, With a view toward algebraic geometry.
  \MR{1322960 (97a:13001)}

\bibitem{MR318138}
D.~Ferrand, \emph{Monomorphismes et morphismes absolument plats}, Bull. Soc.
  Math. France \textbf{100} (1972), 97--128. \MR{318138}

\bibitem{grothendieck:ega4-4}
A.\ Grothendieck, \emph{\'{E}l\'ements de g\'eom\'etrie alg\'ebrique. {IV}.
  \'{E}tude locale des sch\'emas et des morphismes de sch\'emas {IV}}, Inst.
  Hautes \'Etudes Sci. Publ. Math. (1967), no.~32, 361. \MR{0238860 (39 \#220)}

\bibitem{MR0345945}
I.~Kaplansky, \emph{Commutative rings}, revised ed., The University of Chicago
  Press, Chicago, Ill.-London, 1974. \MR{0345945}

\bibitem{matsumura:crt}
H.\ Matsumura, \emph{Commutative ring theory}, second ed., Studies in Advanced
  Mathematics, vol.~8, University Press, Cambridge, 1989. \MR{90i:13001}

\bibitem{nasseh:cfwfa}
S.~Nasseh and K.~Sather-Wagstaff, \emph{Cohen factorizations: weak
  functoriality and applications}, J. Pure Appl. Algebra \textbf{219} (2015),
  no.~3, 622--645. \MR{3279378}

\end{thebibliography}

\providecommand{\bysame}{\leavevmode\hbox to3em{\hrulefill}\thinspace}
\providecommand{\MR}{\relax\ifhmode\unskip\space\fi MR }
\providecommand{\MRhref}[2]{%
  \href{http://www.ams.org/mathscinet-getitem?mr=#1}{#2}
}
\providecommand{\href}[2]{#2}

\end{document}